\documentclass[11pt]{amsart}

\usepackage[T1]{fontenc}
\usepackage{lmodern}
\usepackage[margin=1in]{geometry}
\usepackage{amsmath,amssymb,mathtools}
\usepackage{microtype}
\usepackage{mathtools}
\mathtoolsset{showonlyrefs=true}
\usepackage{xcolor}
\usepackage[
  colorlinks=true,
  linkcolor=blue!48!black,
  citecolor=blue!48!black,
  urlcolor=blue!48!black,
  pdftitle={Fourier-invariant functions with dense zero sets}
]{hyperref}

\numberwithin{equation}{section}

\newtheorem{theorem}{Theorem}[section]

\newtheorem{lemma}[theorem]{Lemma}

\newcommand{\R}{\mathbb R}
\newcommand{\cH}{\mathcal H}
\newcommand{\dd}{\,\mathrm d}
\newcommand{\e}{\mathrm e}
\newcommand{\norm}[1]{\left\lVert #1\right\rVert}
\newcommand{\ip}[2]{\left\langle #1,#2\right\rangle}
\newcommand{\off}{\operatorname{off}}
\DeclareMathOperator{\sech}{sech}

\title[Fourier-invariant functions with dense zero sets] 
{Fourier-invariant functions with dense zero sets}

\begin{document}
\author[A. Bondarenko]{Andriy Bondarenko}
\address{Department of Mathematical Sciences \\ Norwegian University of Science and Technology \\ NO-7491 Trondheim \\ Norway}
\email{andriybond@gmail.com}

\author[K. Seip]{Kristian Seip}
\address{Department of Mathematical Sciences, Norwegian University of Science and Technology (NTNU), 7491 Trondheim, Norway}
\email{kristian.seip@ntnu.no}

\thanks{The authors were supported in part by Grant 334466 of the
Research Council of Norway.}
\begin{abstract}
For every \(0\leq\beta\leq1/2\), we construct a nonzero real-valued continuous 
function $f_\beta$ in $L^1(\mathbb R)\cap L^2(\mathbb R)$
such that \(\widehat {f}_\beta=f_\beta\) and \(f_\beta(\sqrt{n}/[\log(e+n)]^{\beta})=0\) for all \(n\geq0\). The case $\beta=0$ settles in the negative a question raised by Radchenko and Viazovska regarding their Fourier interpolation formula. The construction uses a scale of reproducing kernel Hilbert spaces generated by
the Fourier-invariant Hermite functions. Applying the Mehler
formula, we identify the reproducing kernels of these spaces. By suitable estimates of these kernels, we show that  $(\sqrt{n}/[\log(e+n)]^{\beta})$, with one auxiliary point added to it, is a universal interpolating sequence for at least one of the Hilbert spaces under consideration. However, this result fails when  \(\beta>1/2\).
\end{abstract}

\maketitle

\section{Introduction}

Radchenko and Viazovska proved in \cite{RV} that an even Schwartz function is
determined by its values and the values of its Fourier transform at $
 0,\sqrt1,\sqrt2,\ldots$, 
with the convention for the Fourier transform $\widehat{f}$ of $f$ that
\[
 \widehat f(\xi)
 \coloneqq \int_{\R}f(x)\e^{-2\pi i x\xi}\,\dd x.
\]
In fact, they constructed even Schwartz functions \(a_n\) such that
\[
 h(x)
 =
 \sum_{n=0}^{\infty}a_n(x)h(\sqrt n)
 +\sum_{n=0}^{\infty}\widehat a_n(x)\widehat h(\sqrt n),
\]
with absolute convergence for every even Schwartz function \(h\).  They asked whether this formula continues to hold whenever the right-hand side is well-defined
and absolutely convergent \cite[Question~1]{RV}. A partial positive answer was given in \cite{RV}, namely that the formula still holds with absolute convergence when $h(x)$ and $\widehat{h}(x)$ are both $O((1+|x|)^{-13})$. The latter decay condition was further weakened in \cite[Theorem 7.1]{BRS}.

We will give a negative answer to Radchenko and Viazovska's question. Our main result shows that not only may nontrivial Fourier-invariant  functions vanish on the square-root nodes; they may in fact vanish on considerably denser sequences, namely logarithmic compressions of the form
\[
 \lambda_n(\beta) \coloneqq \frac{\sqrt n}{[\log(\e+n)]^\beta},\quad n\geq0,
\label{eq:nodes}
\]
for $0<\beta \le \tfrac12$.
\begin{theorem}\label{thm:main}
Fix \(0\leq\beta\leq\tfrac12\). There exists a real-valued, continuous, and even function   $f_\beta$ that is both integrable and square-integrable 
such that
\[
 \widehat f_\beta=f_\beta,\qquad
 f_\beta(\lambda_n(\beta))=0\quad(n\geq0),\qquad
 f_\beta(\tfrac12)=1.
\label{eq:main-conclusion}
\]
\end{theorem}

We conclude that there is a limit as to how much we may relax smoothness and decay of the function $h$ in the Radchenko--Viazovska theorem. In a similar vein, we see that the sequence $\Lambda(\beta)\coloneqq (\pm \lambda_n(\beta))$ for $\beta>0$ ``violates'' the incarnation of the uncertainty principle established by Kulikov, Nazarov, and Sodin in \cite{KNS}. More precisely, since
\[
 \lambda_n(\beta)(\lambda_{n+1}(\beta)-\lambda_n(\beta))
 \sim\frac1{2(\log n)^{2\beta}}\longrightarrow0 <\tfrac12, \qquad \beta>0,
\label{eq:supercritical}
\]
$(\Lambda(\beta), \Lambda(\beta))$  is a supercritical pair (for $p=q=2$) in the sense of \cite{KNS} whenever $\beta>0$. Accordingly, by the Kulikov--Nazarov--Sodin uniqueness theorem \cite[Theorem~1.3-UP]{KNS}, a function $f$
with both $f$ and $\widehat{f}$ in the Sobolev space \(H^1(\R)\)  is determined by the values of $f$ and $\widehat{f}$ at the points of $\Lambda(\beta)$. Hence our result shows that  we cannot replace $H^1(\R)$ by $L^1(\R)\cap L^2(\R)$ in that uniqueness theorem.

On the other hand, a classical result of Amrein and Berthier \cite[Proposition 6]{AB} asserts that, given two arbitrary subsets $A$ and $B$ of $\R$, both  of finite Lebesgue measure, there exists a nontrivial $f$ in $L^2(\R)$ such that $f=0$ a.e. on $A$ and $\widehat{f}=0$ a.e. on $B$. Hence the situation becomes drastically different when functions cease to make sense pointwise. One may perhaps think of our result as dealing with an intermediate case in which functions are non-smooth and ``barely'' make sense pointwise. It would be interesting to explore further the finer details of the ``transition'' from the Kulikov--Nazarov--Sodin theorem to that of Amrein--Berthier.

We will obtain Theorem~\ref{thm:main} as a consequence of a stronger result, stated in the next section. This is an interpolation result for  a scale of reproducing kernel Hilbert spaces $\cH_\eta$, with \(\eta>0\) small, consisting of expansions
\[
 g=\sum_{j\geq0}c_j\psi_{4j}, 
\]
where $\psi_n$ denote the Hermite functions. The square of the norm of $g$ in $\cH_\eta$ is comparable to
\begin{equation}
 \sum_{j\geq0}\sqrt{j+1}\,
 [\log(\e+j)]^{1+\eta}|c_j|^2.
\label{eq:intro-norm}
\end{equation}
Every element of \(\cH_\eta\) is even and fixed by the Fourier
transform.  The logarithmic factor in \eqref{eq:intro-norm} makes point evaluation continuous, and it is just large
enough to imply the embedding
\[
 \cH_\eta\subset L^1(\R).
\]
It may be shown that both these properties fail at $\eta=0$. However, it turns out that when $\eta$ is small enough, we are in the most favorable situation and may prove that the sequence $(\lambda_n(\beta))$, along with the additional auxiliary point $\tfrac12$, constitutes a universal interpolating sequence for $\cH_\eta$. Theorem~\ref{thm:main} is an immediate consequence of that result. 

Thanks to Mehler's formula for the Hermite polynomials, we are in Section~\ref{sec:RK} able to get a precise expression for the reproducing kernel $K_\eta(x,y)$ of $\cH_\eta$. The idea is then to appeal to a classical theorem of Bari which asserts that a sequence $(x_n)$, with $x_n\ge 0$, is a universal interpolating sequence for $\cH_\eta$ if and only if the Gram matrix with entries
\[ \frac{K_\eta(x_n,x_m)}{\sqrt{K_\eta(x_n,x_n) K_\eta(x_m,x_m)}} \]
is bounded below and above as an operator on $\ell^2$. The remarkable fact, which is the crux of our proof, is that this Gramian effectively tends to the identity matrix when $\eta\to 0$, in such a way that it has the desired property when $\eta$ is small enough. Our proof of this result is based on a splitting of the kernel into a positive part and an oscillating part; the relevant estimates for each of these parts are carried out in respectively Section~\ref{sec:non} and Section~\ref{sec:osc}. 

The reader may recognize that we could just as well have stated a more general result about sequences with the same density and separation property as $(\lambda_n(\beta))$. To avoid obscuring technicalities, we have chosen not to do that. While it is also possible to replace $(\log x)^\beta$ by somewhat more general weights, we will verify a more curious fact: Our method of proof breaks down when $\beta>\tfrac12$, because the sequence $(\lambda_n(\beta))$ becomes too dense for the Gram matrix to act boundedly on $\ell^2$. We will supply the proof of this fact in the final Section~\ref{sec:conclude}. 

Throughout the paper, $c$ and $C$ will denote positive absolute constants whose values may change from line to line. A subscript, as in $C_\beta$, indicates dependence on some parameter involved in the computations. We write $A\asymp B$ when $cA \le B \le CA$. 

\section{An interpolation theorem for a scale of reproducing kernel Hilbert spaces}

Let \(H_m\) be the physicists' Hermite polynomial, and put
\begin{equation}
 \psi_m(x)=
 \frac{2^{1/4}}{\sqrt{2^m m!}}\,
 H_m(\sqrt{2\pi}\,x)\e^{-\pi x^2}.
 \label{eq:hermite}
\end{equation}
Then \((\psi_m)_{m\geq0}\) is an orthonormal basis of \(L^2(\R)\), and
\[
 \widehat{\psi_m}=(-i)^m\psi_m.
\]
See \cite[Chapter~1]{Thangavelu} for these and other basic facts about the Hermite functions.

Fix
\[
 0<\eta\leq\eta_0,
\]
where \(\eta_0>0\) will be a sufficiently small absolute constant.
For \(j\geq0\), define
\[
 a_j\coloneqq 
 \int_0^{\e^{-2}}
 \frac{\e^{-(j+1)t}}
 {\sqrt t\,[\log(1/t)]^{1+\eta}}\,\dd t.
\]
The following asymptotics will play an essential role.
\begin{lemma}\label{lem:aj}
Uniformly for \(j\geq0\) and \(0<\eta\leq\eta_0<1\),
\[
 a_j\asymp
 \frac1{\sqrt{j+1}\,[\log(\e+j)]^{1+\eta}}.
\]
\end{lemma}

\begin{proof}
We write \(J=j+1\) and put \(s=Jt\).  Then
\[
 a_j=J^{-1/2}
 \int_0^{J\e^{-2}}
 \frac{\e^{-s}s^{-1/2}}
 {[\log(J/s)]^{1+\eta}}\,\dd s.
\]
We see that
\[  \int_{\e^{-2}/\sqrt{J}}^{\e^{-2}\sqrt{J}}
 \frac{\e^{-s}s^{-1/2}}
 {[\log(J/s)]^{1+\eta}}\,\dd s \asymp \frac1{\sqrt{j+1}\,[\log(\e+j)]^{1+\eta}}.
\]
The integral over the remaining intervals $[0,\e^{-2}/\sqrt{J}]$ and $[\e^{-2}\sqrt{J}, e^{-2}J]$ is trivially an order of magnitude smaller, and so the result follows.
\end{proof}

We define the real Hilbert space
\[
 \cH_\eta
 \coloneqq 
 \left\{
 g=\sum_{j=0}^{\infty}c_j\psi_{4j}:
 c_j\in\R,\quad
 \norm{g}_\eta^2
 :=\sum_{j=0}^{\infty}\frac{|c_j|^2}{a_j}<\infty
 \right\},
\]
Since
\(\sup_j a_j<\infty\), 
\[
 \norm{g}_2\leq C\norm{g}_\eta.
\]
Every member of \(\cH_\eta\) is even and is fixed by the Fourier
transform.

\begin{lemma}\label{lem:evaluation}
Point evaluation is continuous on \(\cH_\eta\).  Every
\(g\) in \(\cH_\eta\) has a continuous representative, and the reproducing
kernel of $\cH_\eta$ is
\begin{equation}
 K_\eta(x,y)
 =\sum_{j=0}^{\infty}
 a_j\psi_{4j}(x)\psi_{4j}(y).
\label{eq:kernel-series}
\end{equation}
The series converges locally uniformly on \(\R^2\).
\end{lemma}

\begin{proof}
The estimate
\begin{equation}
 \sup_{x\in E}|\psi_{4j}(x)|
 \leq C_E(j+1)^{-1/4}
\label{eq:fixed-hermite}
\end{equation}
holds for every compact subset \(E\) of \( \R\); see
\cite[\S18.15(v)]{DLMF}.  Hence, by Lemma~\ref{lem:aj},
\begin{equation}
 \sum_{j=0}^{\infty}
 a_j\sup_{x\in E}|\psi_{4j}(x)|^2
 \leq
 C_E\sum_{j=0}^{\infty}
 \frac1{(j+1)[\log(\e+j)]^{1+\eta}}
 <\infty.
\label{eq:evaluation-sum}
\end{equation}
By the Cauchy--Schwarz inequality, we get 
\[
 |g(x)|
 \leq
 \norm{g}_\eta
 \left(
 \sum_{j\geq0}a_j|\psi_{4j}(x)|^2
 \right)^{1/2},
\]
which in view of \eqref{eq:evaluation-sum} ensures uniform convergence of the Hermite expansion of $g$ and hence continuity of $g$.  The reproducing identity and the local uniform convergence of
\eqref{eq:kernel-series} also follow at once.
\end{proof}
Initially, we could only assert that $g=\widehat{g}$ almost everywhere for $g$ in $\cH_\eta$. In view of the preceding lemma, however, we may from now on consider $\cH_\eta$ as consisting of continuous functions, with  
$g(x)=\widehat{g}(x)$ for every real $x$.

We are now ready to state a result of which Theorem~\ref{thm:main} turns out to be a corollary. To this end, we let $\Lambda^\ast(\beta)$ be the sequence consisting of the points $\tfrac12$ and $\lambda_0(\beta), \lambda_1(\beta), \ldots $, and define the operator $R_{\beta,\eta}$ which maps a function $g$ to the sequence
\[ g(x)/\sqrt{K_\eta(x,x)}, \qquad x\in \Lambda^\ast(\beta). \] 
We will let $\ell^2$ denote the Hilbert space of square-summable real-valued sequences indexed by the points   $\tfrac12$ and $\lambda_0(\beta), \lambda_1(\beta), \ldots $.
\begin{theorem}\label{thm:realmain}
Fix $0\le \beta \le \tfrac12$. For all sufficiently small $\eta$, $R_{\beta,\eta}$ maps $\cH_\eta$ into and onto $\ell^2$. 
\end{theorem}
By standard terminology, this means that the sequence $\Lambda^\ast(\beta)$ is a universal interpolating sequence for $\cH_\eta$ once $\eta$ is small enough. Since we may in particular find a function $f_\beta$ in $\cH_\eta$ such that $f_\beta(\tfrac12)=1$ and $f_\beta(\lambda_n)=0$, $n=0,1,\ldots $, we obtain Theorem~\ref{thm:main} from Theorem~\ref{thm:realmain} once the following lemma has been established.

\begin{lemma}\label{prop:l1}
For \(0<\eta\leq\eta_0\), we have
\begin{equation}
 \norm{g}_1\leq C_\eta\norm{g}_\eta
 \qquad g\in\cH_\eta.
\label{eq:l1-embedding}
\end{equation}
\end{lemma}

\begin{proof}
We split $g$ dyadically and write
\[ g_k \coloneqq \sum_{j=2^k}^{2^{k+1}-1} c_j \psi_{4j}, \qquad k\ge 0.\]
The harmonic oscillator
\[
 \mathcal L
 =-\frac1{4\pi}\frac{\dd^2}{\dd x^2}+\pi x^2
\]
satisfies
\[
 \mathcal L\psi_m=(m+1/2)\psi_m.
\]
Consequently,
\[
 \pi\norm{xg_k}_2^2
 \leq\ip{\mathcal Lg_k}{g_k}_{L^2}
 =
 \sum_{2^k\leq j<2^{k+1}}
 (4j+1/2)|c_j|^2
 \leq C2^k\norm{g_k}_2^2,
\]
and hence
\begin{equation}
 \norm{xg_k}_2\leq C2^{k/2}\norm{g_k}_2.
\label{eq:x-block}
\end{equation}
By the Cauchy--Schwarz inequality, we therefore get
 \begin{equation}
 \norm{g_k}_1
 \leq
 \left(
 \int_{\R}\frac{\dd x}{1+2^{-k}x^2}
 \right)^{1/2}
 \left(
 \norm{g_k}_2^2+2^{-k}\norm{xg_k}_2^2
 \right)^{1/2} 
 \leq C2^{k/4}\norm{g_k}_2.
\label{eq:block-l1}
\end{equation}
Lemma~\ref{lem:aj} gives the dyadic norm equivalence
\[
 \norm{g}_\eta^2
 \asymp
 |c_0|^2+
 \sum_{k\geq0}
 2^{k/2}(k+1)^{1+\eta}\norm{g_k}_2^2.
\]
Therefore, by \eqref{eq:block-l1} and another application of the Cauchy--Schwarz inequality,
\[
 \sum_{k\geq0}\norm{g_k}_1 \le 
 C\norm{g}_\eta 
 \left(\sum_{k\geq0}(k+1)^{-1-\eta}\right)^{1/2},
\]
which yields \eqref{eq:l1-embedding} since $\eta>0$.
\end{proof}
By a classical theorem of Bari \cite[p. 132]{N}, a sequence of positive numbers $X=(x_n)$ is a universal interpolating sequence for $\cH_\eta$ if and only if the sequence of 
normalized reproducing kernels
\[ \frac{K_\eta(x_n,\cdot )}{\sqrt{K_\eta(x_n,x_n)}} \]
is a Riesz sequence in $\cH_\eta$. This means that to prove Theorem~\ref{thm:realmain}, we may show that the Gram matrix $G_{\beta,\eta}$ with entries
\[ \frac{K_\eta(x,y)}{\sqrt{K_\eta(x,x)K_\eta(y,y)}}, \qquad x,y\in \Lambda^\ast(\beta), \]
is bounded below and above on $\ell^2$. It follows that Theorem~\ref{thm:realmain} is a consequence of the following theorem.
\begin{theorem}\label{thm:mainmain}
Fix $0\le \beta\le \tfrac12$. For $0<\eta\le \eta_0$, we have
\begin{equation}
 \norm{G_{\beta,\eta}-I}_{\ell^2\to\ell^2}
 \leq
 \begin{cases}
  C_\beta\eta,&0\leq\beta<1/2,\\
  C\sqrt\eta,&\beta=1/2.
 \end{cases}
\label{eq:intro-riesz}
\end{equation}
\end{theorem}
The next three sections will establish Theorem~\ref{thm:mainmain}. 
\section{The reproducing kernel of $\cH_\eta$}\label{sec:RK}

For \(|z|<1\), Mehler's formula in the normalization
\eqref{eq:hermite} is
\begin{equation}
 M_z(x,y)
 \coloneqq 
 \sqrt{\frac2{1-z^2}}\,
 \exp\left[
 -\pi
 \frac{(1+z^2)(x^2+y^2)-4zxy}{1-z^2}
 \right]
 =
 \sum_{m=0}^{\infty}z^m\psi_m(x)\psi_m(y).
\label{eq:mehler}
\end{equation}
The square root is the analytic branch equal to \(\sqrt2\) at
\(z=0\).  See \cite[Eq.~18.18.28]{DLMF}. Mehler's formula yields an explicit expression for $K_\eta(x,y)$, much in the same way as it does in Zelent's computation of the kernel of the space appearing in the Kulikov--Nazarov--Sodin uniqueness theorem \cite{Zelent}.

\begin{lemma}\label{prop:filtered}
For all \(x,y\in\R\),
\begin{equation}
 K_\eta(x,y)
 =
 \frac14
 \int_0^{\e^{-2}}
 \frac{\e^{-t}}
 {\sqrt t\,[\log(1/t)]^{1+\eta}}
 \sum_{\ell=0}^{3}
 M_{i^\ell\e^{-t/4}}(x,y)\,\dd t.
\label{eq:filtered}
\end{equation}
\end{lemma}

\begin{proof}
We start from the fact that
\begin{equation} \label{eq:start}
 \frac14\sum_{\ell=0}^{3}
 M_{i^\ell\e^{-t/4}}(x,y)
 =
 \sum_{j=0}^{\infty}
 \e^{-jt}\psi_{4j}(x)\psi_{4j}(y).
\end{equation}
Multiplying by 
\[ \frac{\e^{-t}}
 {\sqrt t\,[\log(1/t)]^{1+\eta}} \] and integrating, we arrive at \eqref{eq:kernel-series} by interchanging the integral and the sum on the right-hand side of \eqref{eq:start} which is justified by absolute convergence.
\end{proof}

Set \(r\coloneqq \e^{-t/4}\).  We get from \eqref{eq:mehler} that
\begin{equation}
 M_r(x,y)
 =
 \sqrt{\frac2{1-r^2}}\,
 \exp\left[
 -\pi\frac{1+r}{2(1-r)}(x-y)^2
 -\pi\frac{1-r}{2(1+r)}(x+y)^2
 \right],
\label{eq:local-exact}
\end{equation}
and we get $M_{-r}(x,y)$ by interchanging \(x-y\) and
\(x+y\) in this formula.  The two remaining terms satisfy
\begin{equation}
 M_{ir}(x,y)+M_{-ir}(x,y)
 =
 2p(t)\e^{-\pi u(t)(x^2+y^2)}
 \cos(2\pi v(t)xy),
\label{eq:fourier-exact}
\end{equation}
where
\begin{equation}
 p(t)=\left(\frac2{1+\e^{-t/2}}\right)^{1/2},
 \qquad
 u(t)=\tanh(t/4),
 \qquad
 v(t)=\sech(t/4).
\label{eq:puv}
\end{equation}
In particular,
\begin{equation}
 p(t)\asymp1,\qquad
 u(t)\asymp t,\qquad
 0<v_0\leq v(t)\leq1
\label{eq:puv-bounds}
\end{equation}
on \(0<t<\e^{-2}\).

We write
\begin{equation}
 K_\eta
 =K_\eta^{+}
  +K_\eta^{-}
  +K_\eta^{\mathrm{osc}}
\label{eq:branch-decomposition}
\end{equation}
for the three contributions to $K_\eta$ from respectively $M_r$, $M_{-r}$, and $M_{ir}+M_{-ir}$. Here $K_\eta^{\mathrm{osc}}$ is what we will refer to as the oscillating part of $K_\eta$. We see that $K_\eta^{+}
  +K_\eta^{-}$ is positive; we will refer to it as the non-oscillating part of $K_\eta$.

The following estimates will be required. We omit the proof which is based on straightforward computations.
\begin{lemma}\label{lem:scalar-integrals}
For
\(0<\eta\leq\eta_0\) and \(A\geq0\),
\begin{equation}
 \int_0^{\e^{-2}}
 \frac{\e^{-At}}
 {t[\log(1/t)]^{1+\eta}}\,\dd t
 \asymp
 \frac1\eta[\log(\e+A)]^{-\eta},
\label{eq:diag-integral}
\end{equation}
and
\begin{equation}
 \int_0^{\e^{-2}}
 \frac{\e^{-At}}
 {\sqrt t\,[\log(1/t)]^{1+\eta}}\,\dd t
 \leq
 \frac{C}
 {\sqrt{1+A}\,[\log(\e+A)]^{1+\eta}}.
\label{eq:fourier-integral}
\end{equation}
The constants in \eqref{eq:diag-integral} are uniform in
\(A\) and \(\eta\).
\end{lemma}
We may now establish the following bound for the diagonal of $K_\eta$. 
\begin{lemma}\label{prop:diagonal}
If \(\eta_0>0\) is sufficiently small, then
\begin{equation}
 K_\eta(x,x)
 \asymp
 \frac1\eta[\log(\e+x^2)]^{-\eta}
\label{eq:diagonal}
\end{equation}
uniformly for \(x\) in \(\R\) and \(0<\eta\leq\eta_0\).
\end{lemma}

\begin{proof}
For \(0<t<\e^{-2}\), there exist positive constants $c_1, C_1$ and $c_2, C_2$ such that
\begin{equation}
 C_1 t^{-1/2}\e^{-c_1t x^2} \le M_r(x,x) \le 
 C_2 t^{-1/2}\e^{-c_2 t x^2} 
 \label{eq:local-diagonal}
\end{equation}
for all $x$. Since 
\[
 0\leq M_{-r}(x,x)\leq M_r(x,x),
\]
we get on multiplying $M_{-r}(x,x)+ M_r(x,x)$ by the measure in
\eqref{eq:filtered}, integrating, and invoking \eqref{eq:diag-integral} of Lemma~\ref{lem:scalar-integrals},
\begin{equation}
 K_\eta^{+}(x,x)+K_\eta^-(x,x)
 \asymp
 \frac1\eta[\log(\e+x^2)]^{-\eta}.
\label{eq:local-diagonal-size}
\end{equation}
On the other hand, \eqref{eq:fourier-exact},
\eqref{eq:puv-bounds}, and
\eqref{eq:fourier-integral} give
\[
 |K_\eta^{\mathrm{osc}}(x,x)|
 \leq
 \frac{C}
 {\sqrt{1+x^2}\,[\log(\e+x^2)]^{1+\eta}}.
\tag{3.15}\label{eq:osc-diagonal}
\]
The ratio of \eqref{eq:osc-diagonal} to
\eqref{eq:local-diagonal-size} is at most
\[
 \frac{C\eta}
 {\sqrt{1+x^2}\log(\e+x^2)} \le C\eta,
\]
and so we obtain \eqref{eq:diagonal} from \eqref{eq:local-diagonal-size} by choosing $\eta_0 \le 1/(2C)$ with $C$ as in the preceding bound. 
\end{proof}

\section{The non-oscillatory part of the Gramian}\label{sec:non}

Let \(G^{\mathrm{non}}_{\beta,\eta}\) be the non-oscillating part of the Gramian $G_{\beta,\eta}$. In other words, the entries of $G^{\mathrm{non}}_{\beta,\eta}$ are 
\[
 \frac{
 K_\eta^{+}(x,y)
 +K_\eta^{-}(x,y)}
 {\sqrt{K_\eta(x,x)K_\eta(y,y)}}, \qquad x,y\in \Lambda^\ast(\beta).
\]
Only its off-diagonal part $\off G^{\mathrm{non}}_{\beta,\eta}$ is relevant, because the diagonal of the full matrix $G_{\beta,\eta}$ is simply the identity matrix $I$.
The following is the main result of this section.
\begin{lemma}\label{prop:local}
Uniformly for \(0\leq\beta\leq1/2\) and
\(0<\eta\leq\eta_0\),
\begin{equation}
 \norm{\off G^{\mathrm{non}}_{\beta,\eta}}_{\ell^2\to\ell^2}
 \leq C\eta.
\label{eq:local-operator}
\end{equation}
\end{lemma}
To prepare for the proof of Lemma~\ref{prop:local}, we briefly record a few simple facts about the geometry of the nodes $\lambda_n(\beta)$ defined in \eqref{eq:nodes}. We will from now on keep  \(0\leq\beta\leq1/2\) fixed, and we therefore suppress the dependence on $\beta$ by writing \(\lambda_n=\lambda_n(\beta)\). It is convenient
to define also the squared nodes 
\begin{equation}
 q_n\coloneqq \lambda_n^2
 =\frac{n}{[\log(\e+n)]^{2\beta}}.
\label{eq:qn}
\end{equation}
The following is an immediate consequence of the mean value theorem.
\begin{lemma}\label{lem:node-geometry}
For \(n\geq1\),
\begin{equation}
 q_{n+1}-q_n
 \asymp [\log(\e+n)]^{-2\beta},
\label{eq:q-spacing}
\end{equation}
and
\begin{equation}
 \lambda_{n+1}-\lambda_n
 \asymp
 \frac1{\sqrt n\,[\log(\e+n)]^\beta}.
\label{eq:lambda-spacing}
\end{equation}
\end{lemma}
The asymptotics \eqref{eq:q-spacing} implies that if \(n,m\) are in \( [N,2N)\), \(N\geq2\), and
\(L=\log(\e+N)\), then
\begin{equation}
 |q_n-q_m|
 \geq c|n-m|L^{-2\beta}.
\label{eq:dyadic-q}
\end{equation}
For \(R\geq1\), we then  get
\begin{equation}
 \#\{n:R\leq q_n<R+1\}
 \leq C[\log(\e+R)]^{2\beta}
\label{eq:q-density}
\end{equation}
as well.

The estimates
\[
 M_{\e^{-t/4}}(x,y)
 \leq
 Ct^{-1/2}
 \exp\left[
 -c\frac{(x-y)^2}{t}-ct(x+y)^2
 \right]
\]
and
\[
 M_{-\e^{-t/4}}(x,y)
 \leq
 Ct^{-1/2}
 \exp\left[
 -c\frac{(x+y)^2}{t}-ct(x-y)^2
 \right]
\]
follow directly from \eqref{eq:local-exact}.  We will use them only
for \(0<t<\e^{-2}\). In this range,
\[ \frac{(x-y)^2}{t}+t(x+y)^2 \le \frac{(x+y)^2}{t}+t(x-y)^2, \qquad x,y\ge 0, \]
which means that we may combine the two bounds to get
\begin{equation} 
M_{\e^{-t/4}}(x,y)+M_{-\e^{-t/4}}(x,y)
 \leq
 Ct^{-1/2}
 \exp\left[
 -c\frac{(x-y)^2}{t}-ct(x+y)^2
 \right].
\label{eq:local-bound}
\end{equation}
It is convenient to first estimate the size of $K_\eta^{\mathrm{non}}(x,y)$ when $x$ and $y$ are far apart.
\begin{lemma}\label{lem:far-branches}
Suppose that 
\(y^2\) is not in \([x^2/2,2x^2]\).  If
\(\max(x,y)\geq\tfrac12\), then
\begin{equation}
 K_\eta^{\mathrm{non}}(x,y)
 \leq C\e^{-c\max(x^2,y^2)},
\label{eq:far-local-kernel}
\end{equation}
uniformly for \(0<\eta\leq\eta_0\).
\end{lemma}

\begin{proof}
If \(x^2\geq2y^2\), then
\begin{equation}
 \frac{(x-y)^2}{t}
 +t(x+y)^2
 \geq cx^2 (t+t^{-1}).
\label{eq:far-exponent}
\end{equation}
The same statement holds with \(y\) in place of \(x\) when $y^2\ge 2x^2$.  Inserting this into \eqref{eq:filtered} and using
\eqref{eq:local-bound}, we obtain
\[
 |K_\eta^{\mathrm{non}}(x,y)|
 \leq C\int_0^{\e^{-2}}
 \frac{\exp[-cM(t+t^{-1})]}
 {t[\log(1/t)]^{1+\eta}}\,\dd t,
\]
where \(M=\max(x^2,y^2)\geq\tfrac14\).  Since \(\log(1/t)\geq2\), we may
discard the logarithmic denominator.  Then the result follows because the maximum of the remaining integrand is $O(\e^{-c'M})$. \end{proof}

\begin{proof}[Proof of Lemma~\ref{prop:local}]
Our plan is to establish a uniform bound $O(\eta)$ for the row sums of  the matrix
\(\off G^{\mathrm{non}}_{\beta,\eta}\). Since the matrix is symmetric, we get the same bound for the column sums, and so the Schur test will yield \eqref{eq:local-operator}.

The row sums with respect to either $x=0$ or $x=\tfrac12$ are plainly $O(\eta)$ by Lemma~\ref{prop:diagonal} and Lemma~\ref{lem:far-branches}. Similarly, the terms in the remaining row sums associated with the points $y=0$ and $y=\tfrac12$ are uniformly $O(\eta)$ and may be neglected. Hence it suffices to estimate
\[ R_n\coloneqq \sum_{m\neq n} \frac{K^{\mathrm{non}}_\eta(\lambda_m,\lambda_n)}{\sqrt{K_\eta(\lambda_n,\lambda_n)
 K_\eta(\lambda_m,\lambda_m)}}. \]
In fact, again in view of Lemma~\ref{prop:diagonal} and Lemma~\ref{lem:far-branches} along with \eqref{eq:q-density}, we may restrict further to the finite sum
\[ R^{\mathrm{ess}}_n\coloneqq \sum_{\substack{m\neq n\\ q_m\in [q_n/2,2q_n]}} \frac{K^{\mathrm{non}}_\eta(\lambda_m,\lambda_n)}{\sqrt{K_\eta(\lambda_n,\lambda_n)K_\eta(\lambda_m,\lambda_m)}}. \]

We now assume that $q_n/2\le q_m \le 2q_n$  and set \[
 Q\coloneqq  q_n,\qquad
 L=\log(\e+n).
\]
By Lemma~\ref{prop:diagonal} and
Lemma~\ref{lem:node-geometry},
\begin{equation}
 \frac1{\sqrt{K_\eta(\lambda_n,\lambda_n)
 K_\eta(\lambda_m,\lambda_m)}}
 \leq C\eta L^\eta.
\label{eq:local-normalization}
\end{equation}
Moreover,
\begin{equation}
 (\lambda_n-\lambda_m)^2
 =\frac{(q_n-q_m)^2}
 {(\lambda_n+\lambda_m)^2}
 \asymp\frac{(q_n-q_m)^2}{Q}.
 \label{eq:sep}
\end{equation}
Using \eqref{eq:local-normalization} and \eqref{eq:sep} along with \eqref{eq:local-bound} and
\eqref{eq:filtered}, we find that
\[
 R^{\mathrm{ess}}_n \le C\eta L^\eta
 \int_0^{Q\e^{-2}}
 \frac{\e^{-cs}}
 {s[\log(Q/s)]^{1+\eta}}
 \sum_{\substack{m\ne n\\q_m\in [Q/2,2Q]}}
 \exp\left[-c\frac{(q_n-q_m)^2}{s}\right]\dd s.
\tag{4.14}\label{eq:local-row-integral}
\]
By \eqref{eq:dyadic-q},
\[
 \sum_{\substack{m\ne n\\q_m\in [Q/2,2Q]}}
 \exp\left[-c\frac{(q_n-q_m)^2}{s}\right]
 \leq CL^{2\beta}\sqrt s.
\]
For
\(0<s\leq\sqrt Q\), the logarithm in
\eqref{eq:local-row-integral} is bounded below by \(cL\).  The
remaining range is exponentially small because of \(\e^{-cs}\).
Consequently,
\[ R^{\mathrm{ess}}_n \le
 C\eta L^{\eta+2\beta}
 L^{-1-\eta}
 \int_0^\infty s^{-1/2}\e^{-cs}\,\dd s
 \leq C\eta L^{2\beta-1}
 \leq C\eta.
\]
\end{proof}

\section{The oscillatory part of the Gramian}\label{sec:osc}
We turn to  \(G^{\mathrm{osc}}_{\beta,\eta}\), the oscillating part of the Gramian $G_{\beta,\eta}$. Thus the entries of $G^{\mathrm{osc}}_{\beta,\eta}$ are 
\[
 \frac{
 K_\eta^{\mathrm{osc}}(x,y)}
 {\sqrt{K_\eta(x,x)K_\eta(y,y)}}, \qquad x,y\in \Lambda^\ast(\beta).
\]
The following lemma will complete the proof of Theorem~\ref{thm:mainmain}.
\begin{lemma}\label{prop:osc}
For \(0\leq\beta\leq1/2\) and
\(0<\eta\leq\eta_0\),
\begin{equation}
 \norm{G^{\mathrm{osc}}_{\beta,\eta}}_{\ell^2\to\ell^2}
 \leq  \begin{cases}
  C_\beta\eta,&0\leq\beta<1/2,\\
  C\sqrt\eta,&\beta=1/2.
 \end{cases}
\label{eq:osc-operator}
\end{equation}
\end{lemma}
Strictly speaking, we need an estimate for $\| \off G^{\mathrm{osc}}_{\beta,\eta}\|_{\ell^2\to\ell^2}$ to finish the proof of Theorem~\ref{thm:mainmain}. However, it is plain, by \eqref{eq:fourier-exact}, \eqref{eq:fourier-integral}, and
Lemma~\ref{prop:diagonal}, that $\|\operatorname{diag}G^{\mathrm{osc}}_{\beta,\eta}\|_{\ell^2\to\ell^2}=O(\eta)$ for $0\le \beta\le\tfrac12$, so Theorem~\ref{thm:mainmain} does indeed follow from Lemma~\ref{prop:local} and Lemma~\ref{prop:osc}.

We will require Montgomery and Vaughan's version of Hilbert's inequality \cite{MV}. To this end, let
\((\xi_m)\) be a finite sequence of distinct real numbers and write
\[
 \Delta_m=\min_{r\ne m}|\xi_m-\xi_r|.
\]
By \cite[Corollary 2]{MV},
\begin{equation}
 \int_J
 \left|\sum_m b_m\e^{2\pi i\xi_mx}\right|^2\dd x
 \leq
 \left(
 |J|\sum_m|b_m|^2+3\pi
 \sum_m\Delta_m^{-1}|b_m|^2
 \right)
\label{eq:individual-gap}
\end{equation}
for every interval \(J\).  

To prepare for our application of \eqref{eq:individual-gap}, we record the following elementary sampling estimate.

\begin{lemma}\label{lem:sampling}
Let \(x_1,\ldots,x_M\) be \(h\)-separated points in an interval \(J=[a,b]\).
If \(B\) is continuously differentiable, then
\begin{equation}
 \sum_{\nu=1}^M|B(x_\nu)|^2
 \leq
 2\left(
 h^{-1}\int_{J_h}|B(x)|^2\,\dd x
 +h\int_{J_h}|B'(x)|^2\,\dd x
 \right),
\label{eq:sampling}
\end{equation}
where \(J_h\) is the interval $[a-h/2,b+h/2]$.
\end{lemma}

\begin{proof}
The intervals
\((x_\nu-h/2,x_\nu+h/2)\) are disjoint.  On each such interval, the
fundamental theorem of calculus and the Cauchy--Schwarz inequality give
\[
 |B(x_\nu)|^2
 \leq
 2|B(x)|^2+
 2h\int_{x_\nu-h/2}^{x_\nu+h/2}|B'(s)|^2\,\dd s.
\]
We now integrate with respect to $x$ over \((x_\nu-h/2,x_\nu+h/2)\) and divide by \(h\). Finally summing over \(\nu\), we get the desired bound.
\end{proof}
The following large sieve-type inequality will play a crucial role in the proof of Lemma~\ref{prop:osc}. 
\begin{lemma}
\label{lem:large-sieve}
Let \(N\geq2\), \(L\coloneqq \log(\e+N)\), and \(0<v_0\leq v\leq1\).  Then
\begin{equation}
 \begin{aligned}
 \sum_{N\leq n<2N}
 \left|
 \sum_{1\leq m<2N}
 b_m\cos(2\pi  v\lambda_n\lambda_m)
 \right|^2
 \leq C\bigg\{
 &N\sum_{m<2N}|b_m|^2\\
 &+\sqrt N\,L^\beta
 \sum_{m<2N}
 \sqrt{m+1}\,[\log(\e+m)]^\beta|b_m|^2
 \bigg\}.
 \end{aligned}
\label{eq:large-sieve}
\end{equation}
The constant $C$ is uniform for \(0\leq\beta\leq1/2\) and
\(v_0\leq v\leq1\).
\end{lemma}

\begin{proof} We may clearly prove the lemma with $\exp(2\pi i v\lambda_n\lambda_m)$ in place of $\cos(2\pi  v\lambda_n\lambda_m)$.
The points \(\lambda_n\), \(N\leq n<2N\), have separation
\begin{equation}
 h\asymp N^{-1/2}L^{-\beta}
\label{eq:sample-gap}
\end{equation}
and lie in an interval of length $T$ with
\begin{equation}
 T\asymp\sqrt N\,L^{-\beta}.
\label{eq:sample-length}
\end{equation}
We now apply Lemma~\ref{lem:sampling} to
\begin{equation}
 B(x)=\sum_{m<2N}b_m\e^{2\pi i v\lambda_mx}
 \label{eq:BB}
\end{equation}
at \(x=\lambda_n\), and then estimate the right-hand side of \eqref{eq:sampling} by means of \eqref{eq:individual-gap}. The frequencies in \eqref{eq:BB} are \(v\lambda_m\), so their
gaps are \(v\) times the gaps of \(\lambda_m\).  Since \(v\geq v_0\),
\eqref{eq:lambda-spacing} gives uniformly
\begin{equation}
 \Delta_m^{-1}
 \asymp
 \sqrt{m+1}\,[\log(\e+m)]^\beta.
\label{eq:frequency-gap}
\end{equation}
When we apply \eqref{eq:individual-gap} to the right-hand side of \eqref{eq:sampling}, we get four
terms.  The two terms coming from \(B\) are
\[
 h^{-1}T\sum_m|b_m|^2
 \leq CN\sum_m|b_m|^2
\]
and
\[
 h^{-1}\sum_m\Delta_m^{-1}|b_m|^2.
 \]
These two estimates correspond precisely to the two terms on the right-hand side of
\eqref{eq:large-sieve}.  The terms coming from \(B'\) are
\begin{equation}
 hT\sum_m\lambda_m^2|b_m|^2
 \quad\hbox{and}\quad
 h\sum_m\Delta_m^{-1}\lambda_m^2|b_m|^2.
\label{eq:derivative-terms}
\end{equation}
Since \(m<2N\), the coefficient of \(|b_m|^2\) in the first
expression is at most
\[
 hT\lambda_m^2
 \leq
 C L^{-2\beta}
 \frac{m}{[\log(\e+m)]^{2\beta}}
 \leq CN.
\]
For the second expression, \eqref{eq:frequency-gap} gives
\[
 h\Delta_m^{-1}\lambda_m^2
 \leq
 C N^{-1/2}L^{-\beta}
 m^{3/2}[\log(\e+m)]^{-\beta}
 \leq CN.
\]
Thus both expressions in \eqref{eq:derivative-terms} are bounded by
\(CN\sum_m|b_m|^2\), which is the first term on the right-hand side of \eqref{eq:large-sieve}.
\end{proof}

\begin{proof}[Proof of Lemma~\ref{prop:osc}] Rather than using the Schur test as in the preceding section, we will make a direct estimation of the action of $ G_{\beta,\eta}^{\mathrm{osc}}$ on an $\ell^2$ sequence $c$. We begin by noting that, as in the preceding case, we may remove the rows and columns associated with the points $x=0$ and $x=\tfrac12$ because they will trivially yield  a contribution   to the norm of the same size as the right-hand side of \eqref{eq:osc-operator}. For technical convenience, we also remove the row and column associated with $\lambda_1$, which again can be done with at most the same loss.  This means that we aim at estimating the norm of the matrix \(F_{\beta,\eta}\) with entries
\begin{equation}
 F_{\beta,\eta}(n,m)
 \coloneqq 
 \frac{K_\eta^{\mathrm{osc}}(\lambda_n,\lambda_m)}
 {\sqrt{K_\eta(\lambda_n,\lambda_n)
 K_\eta(\lambda_m,\lambda_m)}},
 \qquad n,m\geq2.
\label{eq:F-matrix}
\end{equation}

The matrix $F_{\beta,\eta}$ acts on sequences of the form $c_2, c_3,\ldots$. We may assume without loss of generality that all sequences are finitely supported, so that all sums appearing in the sequel are finite. We let \(P_k\) denote orthogonal projection onto the space of sequences supported on the indices $n$ with
\(2^k\leq n<2^{k+1}\), and so $I=\sum_{k\ge 1} P_k$, where $I$ is the identity map on $\ell^2$. We set for convenience \(N\coloneqq 2^k\), and write
\[
 L\coloneqq \log(\e+N),\qquad L_m\coloneqq \log(\e+m).
\]
Define the lower dyadic part \(F^-\) of $F_{\beta,\eta}$ by retaining in the $n$th row, $2^k\le n <2^{k+1}$, 
only the columns with index \(2\leq m<2N\). It suffices to estimate the norm of $F^-$ because $\| F_{\beta,\eta}\|\le 3\| F^-\|$, which follows from the decomposition
\[ F_{\beta,\eta}c=F^-c+(F^-)^\ast c-\sum_{k=1}^\infty P_k F^- P_k c, \]
where $(F^-)^\ast$ is the transpose of $F^-$ and the last term accounts for the dyadic diagonal blocks of $F_{\beta,\eta}$ appearing in both $F^-$ and $(F^-)^\ast$.

Using \eqref{eq:fourier-exact} and the definition of $K^{\mathrm{osc}}$, we see that
\begin{align} \nonumber  \norm{P_k F^- c}_{\ell^2}& =\frac12\left(\sum_{N\le n < 2N}  \left|
 \int_0^{\e^{-2}}\sum_{2\le m < 2N} c_m
 \frac{\e^{-t} p(t)\e^{-\pi u(t)(q_n+q_m)}
 \cos(2\pi v(t)\lambda_n \lambda_m)} 
 {\sqrt t\,[\log(1/t)]^{1+\eta} \sqrt{K_\eta(\lambda_n,\lambda_n)}\sqrt{K_\eta(\lambda_m,\lambda_m)}}
 \,\dd t \right|^2 \right)^{1/2} \\ 
\le  &  \frac{1}{2} \int_0^{\e^{-2}} \left(\sum_{N\le n < 2N}  \left|
 \sum_{2\le m < 2N} c_m
 \frac{e^{-\pi u(t)(q_n+q_m)}
 \cos(2\pi v(t)\lambda_n \lambda_m)} 
 { \sqrt{K_\eta(\lambda_n,\lambda_n)}\sqrt{K_\eta(\lambda_m,\lambda_m)}}
  \right|^2 \right)^{1/2} \frac{\e^{-t} p(t)}{\sqrt t\,[\log(1/t)]^{1+\eta}} \dd t ,
  \label{eq:minkowski}
  \end{align}
where we in the last step used the following vector-valued version of Minkowski's inequality:
\[ \Big(\sum_{n} \Big(\int_I \rho_n(x)\, \dd x\Big)^2\Big)^{\frac12} \le \int_I \Big(\sum_n \rho^2_n(x) \Big)^{\frac12} \dd x, \qquad \rho_n(x)\ge 0.\]

Recalling Lemma~\ref{prop:diagonal} and applying Lemma~\ref{lem:large-sieve} to \eqref{eq:minkowski}, we then get 
\begin{equation}
 \norm{P_kF^-c}_2
 \leq C\eta L^{\eta/2}
 \int_0^{\e^{-2}} \bigl[N^{1/2}A_t+N^{1/4}L^{\beta/2}B_t\bigr]
 \frac{\e^{-ctN/L^{2\beta}}}
 {\sqrt t[\log(1/t)]^{1+\eta}}\\
 \dd t ,
\label{eq:minkowski-block}
\end{equation}
where 
\[
 \begin{aligned}
 A_t^2&\coloneqq \sum_{1\le m<2N}L_m^\eta
 \e^{-2\pi u(t)q_m}|c_m|^2,\\
 B_t^2&\coloneqq \sum_{1\le m<2N}\sqrt{m+1}\,
 L_m^{\beta+\eta}\e^{-2\pi u(t)q_m}|c_m|^2.
 \end{aligned}
\]
Discarding the factors $\e^{-2\pi u(t)q_m}$ in these two sums, we are left with the bound
\begin{equation}
 \begin{aligned}
 \norm{P_kF^-c}_2
 \leq C\eta L^{\eta/2} I_N\bigg[
 & n^{1/2}
 \left(
 \sum_{1\le m<2N}L_m^\eta|c_m|^2
 \right)^{1/2}\\
 &+N^{1/4}L^{\beta/2}
 \left(
 \sum_{1\le m<2N}
 \sqrt{m+1}\,L_m^{\beta+\eta}|c_m|^2
 \right)^{1/2}
 \bigg], 
 \end{aligned}
\label{eq:block-unsquared}
\end{equation}
where 
\[
 I_N
 \coloneqq 
 \int_0^{\e^{-2}}
 \frac{\e^{-ctN/L^{2\beta}}}
 {\sqrt t\,[\log(1/t)]^{1+\eta}}\,\dd t, \]
 which satisfies the bound 
 \begin{equation} \label{eq:In} I_N  \leq
 CN^{-1/2}L^{\beta-1-\eta} \end{equation}
by \eqref{eq:fourier-integral}.
Inserting \eqref{eq:In} into \eqref{eq:block-unsquared} and squaring, we thus get
\[
 \norm{P_kF^-c}_2^2
 \leq C\eta^2\bigg[
 L^{2\beta-2-\eta}
 \sum_{m<2N}L_m^\eta|c_m|^2
 +N^{-1/2}L^{3\beta-2-\eta}
 \sum_{m<2N}
 \sqrt{m+1}\,L_m^{\beta+\eta}|c_m|^2
 \bigg].
\]
Since $\| F^-c\|_2^2   = \sum_{k=1}^\infty \norm{P_kF^-c}_2^2 $, we then get, after changing the order of summation,
\begin{equation} \label{eq:final}  \| F^-c\|_2^2 \le C \eta^2 (S_1+S_2),  \end{equation}
 where
 \begin{align*} S_1 &\coloneqq \sum_{m=1}^\infty |c_m|^2 [\log (e+m)]^\eta \sum_{k>\frac{\log m}{\log 2}-1} k^{2\beta-2-\eta} \\
 S_2& \coloneqq \sum_{m=1}^\infty 
 |c_m|^2 \sqrt{m+1}\, [\log (e+m)]^{\beta+\eta} \sum_{k>\frac{\log m}{\log 2}-1} 2^{-k/2} k^{3\beta-2-\eta} .
 \end{align*}
 We see that 
 \[ S_1 \le \begin{cases} C_{\beta} \| c\|_2^2, & 0\le \beta<\tfrac12; \\
                                       C \eta^{-1} \| c\|_2^2, &\beta=\tfrac12 \end{cases} \]
and 
\[ S_2 \le C \| c\|_2^2 \]
when $0\le  \beta \le \tfrac12.$ In view of these bounds, the desired result follows from \eqref{eq:final}.
\end{proof}

\section{The case $\beta>\tfrac12$}\label{sec:conclude}

We finally check that $G_{\beta,\eta}$ fails to be bounded on $\ell^2$ when $\beta>\tfrac12$. Here the point is that in this case, there will be too many diagonals close to the main diagonal that contribute in an essential way to the size of $G_{\beta,\eta}c$. To see this, we pick a large positive number $Q$ and consider the sequence $c=(c_n)$ with
\[ c_n=\begin{cases} 1, & Q\le q_n \le Q+1; \\
                                  0, & \text{otherwise}. \end{cases}
\]
Then 
\begin{equation}  \| c \|_{\ell^2}^2\asymp U^{2\beta}, \qquad U\coloneqq \log(\e + Q)  \label{eq:cnorm} \end{equation}
by \eqref{eq:q-spacing} which clearly holds for $\beta>\tfrac12$ as well.
We next estimate
\[ (G_{\beta,\eta}c)_m=\sum_{Q\le q_n \le Q+1} \frac{K_\eta (\lambda_n,\lambda_m)}
 {\sqrt{K_\eta(\lambda_n,\lambda_n)
 K_\eta(\lambda_m,\lambda_m)}}, \qquad Q\le q_m \le Q+1 . \]
A computation using \eqref{eq:fourier-exact} and the integral arising from Lemma~\ref{prop:filtered} shows that
\[ K^{\mathrm{non}}_\eta (\lambda_n,\lambda_m)\ge c U^{-1-\eta}, \qquad Q\le q_n,q_m \le Q+1.\]
On the other hand, using  \eqref{eq:fourier-exact} and \eqref{eq:fourier-integral}, we find that
\[ |K^{\mathrm{osc}}_\eta| (\lambda_n,\lambda_m) \le C Q^{-1/2}U^{-1-\eta}, \qquad Q\le q_n,q_m \le Q+1.\]
Taking into account also Lemma~\ref{prop:diagonal}, we therefore get
\[ |(G_{\beta,\eta}c)_m|\ge c \eta U^{2\beta-1}  \qquad Q\le q_m \le Q+1 , \]
when $Q$ is large enough. Recalling \eqref{eq:cnorm}, we then find that
\[ \frac{\| G_{\beta,\eta} c \|_{\ell^2}}{\| c \|_{\ell^2}} \ge c \eta U^{2\beta-1} \to \infty\]
when $Q\to \infty$, on the assumption that $\beta>\tfrac12$ and $\eta>0$.

We conclude that Theorem~\ref{thm:realmain} fails for $\beta>\frac12$, but the same cannot be said about the weaker assertion of Theorem~\ref{thm:main}. The above computation merely shows that our proof of  
Theorem~\ref{thm:main} does not work for any $\beta>\tfrac12$.
\section*{Acknowledgement} The authors gratefully acknowledge the assistance of OpenAI’s
ChatGPT, whose exploratory input and calculations were essential in
the development of this paper.


\begin{thebibliography}{99}

\bibitem{DLMF}
NIST Digital Library of Mathematical Functions,
\emph{Chapter 18: Orthogonal polynomials},
\url{https://dlmf.nist.gov/18},
accessed 30 July 2026.

\bibitem{AB} 
W.~O.~Amrein and A.~M.~Berthier, 
\emph{On support properties of $L^p$-functions and their Fourier transforms},
J. Funct. Anal. \textbf{24} (1977), 258--267.

\bibitem{BRS}
A. Bondarenko, D. Radchenko, and K. Seip, \emph{Fourier interpolation with zeros of zeta and L-functions},
Constr. Approx. \textbf{57} (2023), 405--461.

\bibitem{KNS}
A.~Kulikov, F.~Nazarov, and M.~Sodin,
\emph{Fourier uniqueness and non-uniqueness pairs},
J. Math. Phys. Anal. Geom. \textbf{21} (2025), 84--130,

\bibitem{MV}
H.~L. Montgomery and R.~C. Vaughan,
\emph{Hilbert's inequality},
J. London Math. Soc. (2) \textbf{8} (1974), 73--82,

\bibitem{N}
N.~K.~Nikolskii,
\emph{Treatise on the Shift Operator.
Spectral Function Theory}. With an appendix by S. V. Hru\u{s}\u{c}ev  and V. V. Peller. Translated from the Russian by Jaak Peetre.
Grundlehren Math. Wiss.[Fundamental Principles of Mathematical Sciences] \textbf{273},
Springer-Verlag, Berlin, 1986. 

\bibitem{RV}
D.~Radchenko and M.~Viazovska,
\emph{Fourier interpolation on the real line},
Publ. Math. Inst. Hautes \'Etudes Sci. \textbf{129} (2019), 51--81,


\bibitem{Thangavelu}
S.~Thangavelu,
\emph{Lectures on Hermite and Laguerre Expansions},
Mathematical Notes, vol.~42,
Princeton University Press, Princeton, 1993.

\bibitem{Zelent}
D.~Zelent,
\emph{Time-frequency localization in the Fourier symmetric Sobolev space}, 
J. Fourier Anal. Appl. \textbf{32} (2026), Paper No. 37, 22 pp.
\end{thebibliography}
\end{document}